\documentclass{article}
\usepackage[%
journal=AdM,
lang=american,
]{ems-journal}

\numberwithin{equation}{section}

\theoremstyle{plain}
\newtheorem{theorem}{Theorem}[section]
\newtheorem{proposition}[theorem]{Proposition}
\newtheorem{lemma}[theorem]{Lemma}
\newtheorem{corollary}[theorem]{Corollary}
\theoremstyle{definition}

\newcommand{\Z}{\mathbb Z}
\newcommand{\Zn}{\mathbb Z_n}
\newcommand{\trsum}[1]{3^{\wedge}#1}
\newcommand{\sigmaS}{\sigma}
\newcommand{\muZ}{\mu_{\Z}}

\begin{document}

\title{The least non-partitionable zero-sum subset for zero-sum triples in finite abelian groups}
\titlemark{Least non-partitionable zero-sum subsets}

\emsauthor{1}{
	\givenname{Yutong}
	\surname{Zhang}
	\mrid{}
	\zblid{}
	\orcid{0009-0000-1220-0702}}{Y.~Zhang}

\emsauthor{2}{
	\givenname{Yaoran}
	\surname{Yang}
	\mrid{}
	\zblid{}
	\orcid{0009-0004-2832-9163}}{Y.~Yang}

\Emsaffil{1}{
	\department{School of Mathematics}
	\organisation{Sichuan University}
	\rorid{}
	\address{}
	\zip{610065}
	\city{Chengdu}
	\country{China}
	\affemail{yutongzhang@stu.scu.edu.cn}}

\Emsaffil{2}{
	\department{School of Mathematics}
	\organisation{Sichuan University}
	\rorid{}
	\address{}
	\zip{610065}
	\city{Chengdu}
	\country{China}
	\affemail{yangyaoran@stu.scu.edu.cn}}

\classification[20K01, 05C70]{11B75}

\begin{abstract}
Let \(G\) be a finite abelian group, and let \(\mu(G)\) denote the least size of a subset \(S\subseteq G\) with \(3\mid |S|\), total sum zero, and no partition into zero-sum triples; put \(\mu(G)=\infty\) if no such subset exists. We prove the exact classification \(\mu(G)=\infty\) precisely for groups of order at most \(8\) and for \(G\cong C_3^2\), while \(\mu(G)=6\) for every other finite abelian group. The cyclic special case gives \(\mu(\mathbb Z_n)=\infty\) for \(1\le n\le 8\) and \(\mu(\mathbb Z_n)=6\) for \(n\ge 9\), answering the corresponding non-partite cyclic question. We also record a higher-uniformity interval construction which explains the large cyclic witnesses as the case \(k=3\) of a general \(k\)-tuple obstruction.
\end{abstract}

\maketitle

\section{Introduction}

Let \(G\) be an additive abelian group. A \emph{zero-sum triple partition} of a finite subset \(S\subseteq G\) is a decomposition
\[
  S=T_1\sqcup T_2\sqcup\cdots\sqcup T_m,
  \qquad |T_i|=3,
  \qquad \sum_{t\in T_i}t=0\quad (1\le i\le m),
\]
where \(m\ge0\); for \(m=0\) this is the empty partition of the empty set. For a finite subset \(S\) of an additive abelian group \(G\), write
\[
  \sigmaS(S)=\sum_{s\in S}s,
  \qquad
  \trsum{S}=\{x+y+z\mid x,y,z\in S,\ |\{x,y,z\}|=3\}\subseteq G,
\]
where \(\sigmaS(S)\) is evaluated in \(G\).
Call \(S\) \emph{admissible} if \(3\mid |S|\) and \(\sigmaS(S)=0\). For a finite abelian group \(G\), define \(\mu(G)\) to be the least cardinality of a non-partitionable admissible subset of \(G\), and put \(\mu(G)=\infty\) if no such subset exists. For cyclic groups we write
\[
  \muZ(n)=\mu(\Zn),
\]
where \(\Zn\) denotes the cyclic group of order \(n\).

In \cite[Problem 7.4]{MuyesserPokrovskiy}, after asking for the asymptotics of a partite cyclic matching threshold \(g(n)\), M\"uyesser and Pokrovskiy ask the following non-partite variant:
\begin{quote}
What is the smallest subset \(S\subseteq \Z_n\) with \(|S|\) divisible by \(3\), \(\sum S=0\), and \(S\) not partitionable into zero-sum triples?
\end{quote}
The cyclic case of the theorem below answers that question exactly. The same method gives a complete finite abelian group classification, with one additional noncyclic exceptional group.

The problem is connected with complete mappings and group partitions, beginning with Hall--Paige \cite{HallPaige} and Friedlander, Gordon, and Tannenbaum \cite{FriedlanderGordonTannenbaum}, and continuing in recent work such as \cite{CichaczCompleteMappings}. Zero-sum partition questions for finite abelian groups go back to Tannenbaum \cite{Tannenbaum1981,Tannenbaum1983}; see also \cite{KaplanLevRoditty,Zeng,CichaczSameOrder,CichaczTwoPower,CichaczZSP,CichaczSurvey,LladoMoragas,GaoGeroldinger}. The obstruction \(0\notin\trsum{S}\) is a restricted-sumset condition; relevant references include \cite{DiasHamidoune,AlonNathansonRuzsa,GallardoGrekosHabsiegerHennecartLandreauPlagne,Lev,BajnokCritical,BajnokCorrigendum,BajnokEdwards}. The related \(k\)-Harborth constant perspective appears in \cite{LemosMoriyaMouraSilva}.

The main result is the following closed formula.

\begin{theorem}\label{thm:finite-abelian-main}
Let \(G\) be a finite abelian group. Then
\[
  \mu(G)=
  \begin{cases}
  \infty, & |G|\le 8\text{ or }G\cong C_3^2,\\[1mm]
  6, & \text{otherwise}.
  \end{cases}
\]
\end{theorem}

For cyclic groups this gives the exact answer to the question quoted above.

\begin{corollary}\label{cor:cyclic-main}
For every integer \(n\ge1\),
\[
  \muZ(n)=
  \begin{cases}
  \infty,&1\le n\le8,\\[1mm]
  6,&n\ge9.
  \end{cases}
\]
Equivalently, every admissible subset of \(\Zn\) is partitionable for \(1\le n\le8\), while for each \(n\ge9\) there is a six-element admissible subset of \(\Zn\) with no zero-sum triple partition.
\end{corollary}

The proof has three ingredients. First, no obstruction can have size less than \(6\), and a six-element zero-sum subset is non-partitionable exactly when it contains no zero-sum triple. Second, explicit cyclic witnesses exist from order \(9\) onward; for \(n\ge15\) these are instances of a general \(k\)-uniform interval construction, and the six orders \(9\le n\le14\) are covered by finite certificates. Third, when a finite abelian group has no element of order at least \(9\), the structure theorem leaves only bounded-exponent cases; these are handled by a short list of model subgroups and by the two exceptional families in Theorem~\ref{thm:finite-abelian-main}.

\section{General reductions}

\begin{lemma}\label{lem:lower-bound-six-general}
Let \(G\) be an abelian group. Any non-partitionable admissible subset of \(G\) has cardinality at least \(6\). Hence, whenever such a subset of cardinality \(6\) exists, it is automatically of minimum possible size.
\end{lemma}

\begin{proof}
The empty set is partitionable by the empty partition. If \(|S|=3\) and \(\sigmaS(S)=0\), then \(S\) itself is a zero-sum triple, hence \(S\) is partitioned by the single block \(S\). Thus no admissible subset of cardinality \(0\) or \(3\) can be a counterexample. Since admissible cardinalities are multiples of \(3\), the first possible non-partitionable cardinality is \(6\).
\end{proof}

\begin{lemma}\label{lem:six-reduction}
Let \(G\) be an abelian group, and let \(S\subseteq G\) satisfy
\[
  |S|=6,
  \qquad \sigmaS(S)=0.
\]
Then \(S\) has a zero-sum triple partition if and only if
\[
  0\in\trsum{S}.
\]
Consequently, any six-element subset \(S\subseteq G\) satisfying \(\sigmaS(S)=0\) and \(0\notin\trsum{S}\) is a non-partitionable admissible set.
\end{lemma}

\begin{proof}
If
\[
  S=T_1\sqcup T_2,
  \qquad |T_1|=|T_2|=3,
  \qquad \sigmaS(T_1)=\sigmaS(T_2)=0,
\]
then \(T_1\) is a three-element subset of \(S\) with sum zero, so \(0\in\trsum{S}\). Conversely, suppose \(T\subseteq S\), \(|T|=3\), and \(\sigmaS(T)=0\). Then
\[
  \sigmaS(S\setminus T)=\sigmaS(S)-\sigmaS(T)=0.
\]
Since \(|S\setminus T|=3\), the decomposition
\[
  S=T\sqcup(S\setminus T)
\]
forms a zero-sum triple partition. This proves the equivalence.
\end{proof}

\begin{lemma}\label{lem:subgroup-inheritance}
Let \(H\) be a subgroup of an abelian group \(G\). If \(H\) contains a six-element subset \(S\) with \(\sigmaS(S)=0\) and \(0\notin\trsum{S}\), then \(\mu(G)=6\).
\end{lemma}

\begin{proof}
The same subset \(S\) is a subset of \(G\), its total sum is still zero, and a triple of elements of \(S\) sums to zero in \(G\) if and only if it sums to zero in \(H\). Thus \(S\) is a six-element non-partitionable admissible subset of \(G\) by Lemma~\ref{lem:six-reduction}. Lemma~\ref{lem:lower-bound-six-general} gives minimality.
\end{proof}

\section{Cyclic witnesses and a higher-uniformity construction}

For \(k\ge2\), define \(\mu_k(\Zn)\) to be the least cardinality of a subset \(S\subseteq\Zn\) such that \(k\mid |S|\), \(\sigmaS(S)=0\), and \(S\) admits no partition into zero-sum \(k\)-subsets. If no such subset exists, put \(\mu_k(\Zn)=\infty\).

\begin{proposition}\label{prop:k-uniform}
For every \(k\ge2\) and every \(n\ge k(2k-1)\),
\[
  \mu_k(\Zn)=2k.
\]
\end{proposition}

\begin{proof}
No admissible set of cardinality \(0\) or \(k\) can be a counterexample: the empty set has the empty partition, and a \(k\)-element admissible set is itself a zero-sum \(k\)-block. Hence every counterexample has cardinality at least \(2k\).

Set
\[
  M=0+1+\cdots+(2k-2)=(k-1)(2k-1),
\]
and define
\[
  S=\{0,1,2,\ldots,2k-2,-M\}\subseteq\Zn.
\]
Since \(n\ge k(2k-1)\), the interval \(\{0,1,\ldots,2k-2\}\) consists of distinct residues modulo \(n\). Also \(-M\not\equiv j\pmod n\) for \(0\le j\le 2k-2\), because then \(n\) would divide \(M+j\), whereas
\[
  0<M+j\le M+2k-2=k(2k-1)-1<n.
\]
Thus \(|S|=2k\), and \(\sigmaS(S)=0\).

We claim that no \(k\)-element subset of \(S\) has sum zero. First suppose that a \(k\)-element subset \(T\) does not contain \(-M\). Then its integer sum is positive and at most
\[
  (2k-2)+(2k-3)+\cdots+(k-1)=\frac{3k(k-1)}{2}<k(2k-1)\le n,
\]
so \(\sigmaS(T)\not\equiv0\pmod n\).

Now suppose that \(-M\in T\). Write \(T=\{-M\}\cup U\), where \(U\subseteq\{0,1,\ldots,2k-2\}\) and \(|U|=k-1\). The integer sum of the elements of \(U\) is at most
\[
  (2k-2)+(2k-3)+\cdots+k=\frac{(3k-2)(k-1)}{2}<M,
\]
so the integer sum of \(T\) is negative. Its absolute value is at most
\[
  M-(0+1+\cdots+(k-2))
  =(k-1)(2k-1)-\frac{(k-1)(k-2)}{2}
  =\frac{3k(k-1)}{2}<n.
\]
Therefore \(\sigmaS(T)\not\equiv0\pmod n\) in this case as well.

Thus \(S\) has size \(2k\), total sum zero, and no zero-sum \(k\)-subset. In particular it admits no partition into zero-sum \(k\)-subsets. Together with the lower bound, this gives \(\mu_k(\Zn)=2k\).
\end{proof}

For \(k=3\), Proposition~\ref{prop:k-uniform} gives the following cyclic witnesses for all sufficiently large orders.

\begin{corollary}\label{cor:large-cyclic-witness}
For every \(n\ge15\), the subset
\[
  \{0,1,2,3,4,-10\}\subseteq\Zn
\]
has cardinality \(6\), total sum zero, and no zero-sum three-element subset.
\end{corollary}

The remaining cyclic orders \(9\le n\le14\) are covered by the finite certificate in Tables~\ref{tab:small-witnesses} and~\ref{tab:small-sum-certificates}. In Table~\ref{tab:small-witnesses}, \(\trsum{S}\) denotes the restricted three-fold sumset modulo \(n\).

\begin{table}[ht]
\centering
\begin{tabular}{cccc}
\toprule
\(n\) & \(S\subseteq\Zn\) & \(\sigmaS(S)\) & \(\trsum{S}\) \\
\midrule
\(9\)  & \(\{1,2,4,5,7,8\}\)     & \(0\) & \(\Zn\setminus\{0\}\) \\
\(10\) & \(\{1,2,4,6,8,9\}\)     & \(0\) & \(\Zn\setminus\{0\}\) \\
\(11\) & \(\{0,1,3,4,5,9\}\)     & \(0\) & \(\Zn\setminus\{0\}\) \\
\(12\) & \(\{0,1,2,6,7,8\}\)     & \(0\) & \(\Zn\setminus\{0,6\}\) \\
\(13\) & \(\{0,1,2,6,8,9\}\)     & \(0\) & \(\Zn\setminus\{0\}\) \\
\(14\) & \(\{0,1,2,6,9,10\}\)    & \(0\) & \(\Zn\setminus\{0\}\) \\
\bottomrule
\end{tabular}
\caption{Six-element zero-sum witnesses for \(9\le n\le14\).}
\label{tab:small-witnesses}
\end{table}

For auditability, Table~\ref{tab:small-sum-certificates} lists all \(\binom63=20\) restricted triple sums for each row of Table~\ref{tab:small-witnesses}. If \(S=\{s_1<s_2<\cdots<s_6\}\), the sequence \(R_n(S)\) is ordered lexicographically by triples \((i,j,k)\) with \(1\le i<j<k\le6\):
\[
  R_n(S)=\bigl(s_i+s_j+s_k \bmod n\bigr)_{1\le i<j<k\le6}.
\]

\begin{table}[ht]
\centering
\small
\begin{tabular}{c@{\quad}l@{\qquad}c@{\quad}l}
\toprule
\(n\) & \(R_n(S)\) & \(n\) & \(R_n(S)\) \\
\midrule
\(9\) &
\(\begin{gathered}(7,8,1,2,1,3,4,4,5,7,\\ 2,4,5,5,6,8,7,8,1,2)\end{gathered}\)
&
\(10\) &
\(\begin{gathered}(7,9,1,2,1,3,4,5,6,8,\\ 2,4,5,6,7,9,8,9,1,3)\end{gathered}\)
\\
\(11\) &
\(\begin{gathered}(4,5,6,10,7,8,1,9,2,3,\\ 8,9,2,10,3,4,1,5,6,7)\end{gathered}\)
&
\(12\) &
\(\begin{gathered}(3,7,8,9,8,9,10,1,2,3,\\ 9,10,11,2,3,4,3,4,5,9)\end{gathered}\)
\\
\(13\) &
\(\begin{gathered}(3,7,9,10,8,10,11,1,2,4,\\ 9,11,12,2,3,5,3,4,6,10)\end{gathered}\)
&
\(14\) &
\(\begin{gathered}(3,7,10,11,8,11,12,1,2,5,\\ 9,12,13,2,3,6,3,4,7,11)\end{gathered}\)
\\
\bottomrule
\end{tabular}
\caption{Complete residue certificates for the restricted triple sums in Table~\ref{tab:small-witnesses}.}
\label{tab:small-sum-certificates}
\end{table}

\begin{lemma}\label{lem:cyclic-witness-all}
For every \(n\ge9\), the cyclic group \(\Zn\) contains a six-element subset \(S\) with \(\sigmaS(S)=0\) and \(0\notin\trsum{S}\).
\end{lemma}

\begin{proof}
For \(n\ge15\), use Corollary~\ref{cor:large-cyclic-witness}. For \(9\le n\le14\), take the set \(S\) displayed in Table~\ref{tab:small-witnesses}. The six displayed representatives are pairwise distinct modulo \(n\), and direct addition gives \(\sigmaS(S)=0\), as recorded in the third column of the table. Table~\ref{tab:small-sum-certificates} lists all \(20\) restricted triple sums modulo \(n\). In each row, \(0\) is absent. Hence \(0\notin\trsum{S}\) for every \(9\le n\le14\).
\end{proof}

\section{Model witnesses in bounded exponent groups}

The next lemma supplies the noncyclic witnesses needed after the cyclic subgroups have been removed. In the table, the displayed generators have the orders indicated by the corresponding direct-sum factors.

\begin{lemma}\label{lem:model-witnesses}
Each of the following finite abelian groups contains a six-element subset \(S\) with \(\sigmaS(S)=0\) and \(0\notin\trsum{S}\).
\[
\begin{array}{c|c}
\text{group} & S \\
\hline
C_5^2=\langle a,b\rangle
& \{b,2b,3b,4b,a,-a\} \\[1mm]
C_p^2=\langle a,b\rangle,\ p\ge7
& \{0,b,2b,3b,a,-a-6b\} \\[1mm]
C_3^3=\langle a,c,d\rangle
& \{0,c,d,c+d,a,-a+c+d\} \\[1mm]
C_2^4=\langle e_1,e_2,e_3,e_4\rangle
& \{0,e_1,e_2,e_3,e_4,e_1+e_2+e_3+e_4\} \\[1mm]
C_8\oplus C_2=\langle a\rangle\oplus\langle b\rangle
& \{0,b,a,a+b,2a,4a\} \\[1mm]
C_4^2=\langle a,b\rangle
& \{0,b,2b,a,a+b,2a\} \\[1mm]
C_4\oplus C_2^2=\langle a\rangle\oplus\langle b,c\rangle
& \{0,c,b,a,a+c,2a+b\} \\[1mm]
C_2^2\oplus C_3=\langle a,b\rangle\oplus\langle c\rangle
& \{0,c,b,b+c,a,a+c\} \\[1mm]
C_2\oplus C_3^2=\langle a\rangle\oplus\langle c,d\rangle
& \{0,c,d,c+d,a,a+c+d\}.
\end{array}
\]
\end{lemma}

\begin{proof}
In every row the displayed elements are pairwise distinct and have total sum zero. We verify that no three of them sum to zero.

First consider the two rows \(C_p^2\). For \(p=5\), let \(Q=\{1,2,3,4\}\subseteq C_5\); for \(p\ge7\), let \(Q=\{0,1,2,3\}\subseteq C_p\). In both cases no three distinct elements of \(Q\) sum to zero, and \(q_0:=\sum_{q\in Q}q\) does not lie in \(Q\). The displayed subset has the form
\[
  S=(\{0\}\times Q)\cup\{(1,0),(-1,-q_0)\}\subseteq C_p^2.
\]
A zero-sum triple using exactly one off-line point has nonzero first coordinate. A zero-sum triple using no off-line point would give three distinct elements of \(Q\) summing to zero. A zero-sum triple using both off-line points would require the remaining line point to be \(q_0\), which is not in \(Q\). Thus \(0\notin\trsum{S}\).

For the row \(C_3^3\), put \(Q=\{0,c,d,c+d\}\subseteq\langle c,d\rangle\). No three distinct elements of \(Q\) sum to zero. Projecting to the \(a\)-coordinate, a zero-sum triple in
\[
  \{0,c,d,c+d,a,-a+c+d\}
\]
must either use three points of \(Q\), or use the two off-line points \(a\) and \(-a+c+d\) together with one point of \(Q\). The first case is impossible by the choice of \(Q\). In the second case the point of \(Q\) would have to be \(-c-d=2c+2d\), which is not in \(Q\). The same argument applies to the row \(C_2\oplus C_3^2\): the \(C_2\)-coordinate forces a zero-sum triple either to use no elements from the nonzero \(C_2\)-layer or to use both of them; the latter case again requires the missing point \(2c+2d\notin Q\).

For \(C_2^2\oplus C_3\), write
\[
  S=\{0,c,b,b+c,a,a+c\}.
\]
The \(C_2^2\)-projection shows that a zero-sum triple cannot use one point from the \(a\)-layer and one point from the \(b\)-layer, since no point of the \(a+b\)-layer is displayed. It also cannot use three points from the zero layer, which contains only \(0\) and \(c\). Hence it would have to consist of the two displayed elements from one nonzero layer, either \(a+\langle c\rangle\) or \(b+\langle c\rangle\), together with one element from the zero layer. The two displayed elements in such a nonzero layer have \(C_3\)-coordinate sum \(c\), so the zero-layer element would need \(C_3\)-coordinate \(-c=2c\), but no such zero-layer element is displayed.

It remains to check the \(2\)-primary rows. In \(C_8\oplus C_2\), a triple with exactly one element of nonzero \(b\)-coordinate cannot sum to zero. A triple with no nonzero \(b\)-coordinate is a three-element subset of \(\{0,a,2a,4a\}\), whose possible sums are not \(0\) in \(C_8\). A triple using both \(b\) and \(a+b\) would need a remaining element equal to \(-a=7a\), which is not displayed.

For \(C_4^2\), the displayed elements \(0,b,2b,a,a+b,2a\) have first coordinates
\[
  0,0,0,1,1,2.
\]
The only possible first-coordinate patterns summing to \(0\) modulo \(4\) are \(0+0+0\) and \(1+1+2\). In the first case the second-coordinate sum is \(3b\ne0\), and in the second it is \(b\ne0\). The row \(C_4\oplus C_2^2\) is identical at the level of the first coordinate: the possible patterns are again \(0+0+0\) and \(1+1+2\), and the remaining \(C_2^2\)-coordinate sum is \(b+c\) in both cases.

Finally, in \(C_2^4\), write \(u=e_1+e_2+e_3+e_4\). A triple containing \(0\) cannot sum to zero, since the other two elements would have to be equal. A triple not containing \(0\) is either the sum of three distinct basis vectors, which has Hamming weight \(3\), or is of the form \(u+e_i+e_j\) with \(i\ne j\), which has Hamming weight \(2\). Hence it is never zero.
\end{proof}

\section{The exceptional groups}

\begin{lemma}\label{lem:exceptional-groups}
Let \(G\) be a finite abelian group with \(|G|\le8\), or let \(G\cong C_3^2\). Then every admissible subset of \(G\) has a zero-sum triple partition. Hence \(\mu(G)=\infty\).
\end{lemma}

\begin{proof}
As in Lemma~\ref{lem:lower-bound-six-general}, admissible subsets of cardinality \(0\) or \(3\) are partitionable. If \(|G|\le5\), there are no larger admissible cardinalities, so the assertion is immediate. It remains to inspect the abelian groups of orders \(6\), \(7\), and \(8\), and then \(C_3^2\).

The only abelian group of order \(6\) is \(C_6\). Its only six-element subset is the whole group, but
\[
  \sigmaS(C_6)=0+1+2+3+4+5=15\equiv3\pmod6,
\]
so no six-element admissible subset exists. The only abelian group of order \(7\) is \(C_7\). If \(S=C_7\setminus\{c\}\) and \(\sigmaS(S)=0\), then \(c=0\), since \(\sigmaS(C_7)=0\). Thus
\[
  S=\{1,2,3,4,5,6\}=\{1,2,4\}\sqcup\{3,5,6\},
\]
and both displayed triples have sum zero modulo \(7\).

There are three abelian groups of order \(8\). For \(C_8\), write \(S=C_8\setminus\{a,b\}\) with \(a\ne b\). Since \(\sigmaS(C_8)=28\equiv4\pmod8\), the condition \(\sigmaS(S)=0\) gives \(a+b\equiv4\pmod8\). The unordered pairs of distinct residues satisfying this congruence are exactly
\[
  \{0,4\},\quad \{1,3\},\quad \{5,7\}.
\]
In these three cases, respectively, \(S\) contains the zero-sum triples
\[
  \{1,2,5\},\qquad \{0,2,6\},\qquad \{0,2,6\}.
\]
Since \(|S|=6\) and \(\sigmaS(S)=0\), the complement of the displayed zero-sum triple is also a zero-sum triple, so \(S\) is partitionable.

For \(C_2^3\), the sum of all group elements is zero. Every six-element subset is the complement of a two-element set \(\{x,y\}\), and its total sum is \(x+y\). This is zero only if \(x=y\), which is impossible for a two-element set. Hence no six-element admissible subset exists.

For \(C_4\oplus C_2\), write \(G=\langle a\rangle\oplus\langle b\rangle\), with \(a\) of order \(4\) and \(b\) of order \(2\). Again \(\sigmaS(G)=0\). If \(S=G\setminus\{x,y\}\) is admissible of size \(6\), then \(x+y=0\), so \(\{x,y\}\) is either \(\{a,-a\}\) or \(\{a+b,-a+b\}\). In the first case the remaining six elements are partitioned as
\[
  \{0,a+b,-a+b\}\sqcup\{b,2a,2a+b\},
\]
and in the second case as
\[
  \{0,a,-a\}\sqcup\{b,2a,2a+b\}.
\]
Each displayed triple has sum zero.

Finally let \(G=C_3^2\). The whole group has total sum zero and is partitioned into three parallel affine lines. It remains to consider a six-element admissible subset \(S\). Its complement \(T=G\setminus S\) has cardinality \(3\) and total sum zero. In \(C_3^2\), every three-element zero-sum subset is an affine line: if \(T=\{x,y,z\}\) and \(x+y+z=0\), then with \(d=y-x\) one has \(z=x-d\), so \(T=x+\langle d\rangle\). The complement of an affine line in \(C_3^2\) is the union of the two other parallel affine lines, and each affine line has sum zero. Hence \(S\) is partitionable into two zero-sum triples.
\end{proof}

\section{Proof of the classification}

\begin{proof}[Proof of Theorem~\ref{thm:finite-abelian-main}]
The exceptional cases are exactly Lemma~\ref{lem:exceptional-groups}. Conversely, let \(G\) be a finite abelian group with \(|G|>8\) and \(G\not\cong C_3^2\). We show that \(G\) contains one of the witness subgroups from Lemma~\ref{lem:cyclic-witness-all} or Lemma~\ref{lem:model-witnesses}. Lemma~\ref{lem:subgroup-inheritance} will then give \(\mu(G)=6\).

If \(G\) contains an element of order \(m\ge9\), then \(G\) contains a cyclic subgroup \(C_m\), and Lemma~\ref{lem:cyclic-witness-all} applies. Since the exponent of a finite abelian group is attained as the order of an element, we may therefore assume that the exponent of \(G\) is at most \(8\). By the structure theorem for finite abelian groups, the following cases exhaust all possibilities.

If the exponent is \(2\), then \(G\cong C_2^r\). Since \(|G|>8\), we have \(r\ge4\), so \(G\) contains \(C_2^4\).

If the exponent is \(3\), then \(G\cong C_3^r\). The cases \(r=0,1\) have order at most \(3\), and the case \(r=2\) is the excluded group \(C_3^2\). Thus \(r\ge3\), and \(G\) contains \(C_3^3\).

If the exponent is \(4\), then
\[
  G\cong C_4^a\oplus C_2^b,
  \qquad a\ge1.
\]
If \(a\ge2\), then \(G\) contains \(C_4^2\). If \(a=1\) and \(b\ge2\), then \(G\) contains \(C_4\oplus C_2^2\). The remaining cases \(C_4\) and \(C_4\oplus C_2\) have order at most \(8\).

If the exponent is \(5\) or \(7\), then \(G\cong C_p^r\) with \(p\in\{5,7\}\). The case \(r=1\) has order at most \(7\), while \(r\ge2\) gives a subgroup \(C_p^2\).

If the exponent is \(6\), then
\[
  G\cong C_2^a\oplus C_3^b,
  \qquad a,b\ge1.
\]
The case \(a=b=1\) is cyclic of order \(6\). If \(a\ge2\), then \(G\) contains \(C_2^2\oplus C_3\). If \(a=1\) and \(b\ge2\), then \(G\) contains \(C_2\oplus C_3^2\).

If the exponent is \(8\), then
\[
  G\cong C_8^a\oplus C_4^b\oplus C_2^c,
  \qquad a\ge1.
\]
The case \(G\cong C_8\) has order \(8\). In every other case, \(G\) contains \(C_8\oplus C_2\), by taking an element of order \(2\) from one of the remaining direct factors.

Thus every non-exceptional finite abelian group contains one of the required witness subgroups. This proves \(\mu(G)=6\) in all non-exceptional cases and completes the proof.
\end{proof}

\begin{proof}[Proof of Corollary~\ref{cor:cyclic-main}]
A cyclic group is never isomorphic to \(C_3^2\). Hence Theorem~\ref{thm:finite-abelian-main} gives \(\muZ(n)=\infty\) for \(1\le n\le8\), and \(\muZ(n)=6\) for \(n\ge9\).
\end{proof}
\section*{Declaration on the use of AI tools}

The author used GPT-5.5 Pro as an auxiliary tool during the preparation of this work, specifically for enumerating cases, generating possible directions for counterexample searches, and identifying potentially relevant literature. All AI-assisted outputs were independently reviewed, verified, and revised by the author. All references cited in this work were checked against their original sources. The mathematical arguments, conclusions, and final form of the manuscript are the sole responsibility of the author.

\end{document}